\documentclass{article}
\usepackage{amssymb}
\usepackage{amsmath}
\usepackage{amsthm}
 \newcommand{\nn}{\mathbb N}
 \newcommand{\rr}{\mathbb R}

\newtheorem{thm}{Theorem}
\newtheorem{coro}{Corollary}

\begin{document}
\normalsize

\title{Product measures without utilizing integrals}
\author{Hasan G\"ul\footnote{Middle East Technical University, 
e032209@metu.edu.tr} \ and S\"uleyman \"Onal\footnote{Middle East Technical University, osul@metu.edu.tr}}

\date{13.05.2016}
\maketitle

\noindent
{\bf Abstract:} {\normalsize
We show that product measures can be handled without resorting to integrals. 
}\\


\section{Introduction}
In many textbooks on Real Analysis or Measure Theory, if not in all, product measures on product spaces are handled by utilizing the monotone and dominated convergence theorems which involve integrals with respect to some measures   (cf. \cite{AB, HS, LA, RO, RU}). Of course, we do not find this approach to be circular, but feel a certain uneasiness.

The necessary foundation for studying a measure space $ (X, \Sigma, \mu) $ is the characterization of null sets with respect to $ \mu $ and the establishment of $ \sigma$-additivity of it on the elements of the semiring ($ \sigma$-algebra) $ \Sigma $ .

In this short note, we show that an introduction to product measure and laying the aforementioned foundation without resorting to integrals is possible in a very simple way. 

\section{Preliminaries}
In this section, we recall some basic definitions.

Let $ X $ be a nonempty set.

\textbf{Definition 2.1.}   A collection $ \Sigma $ of subsets of $ X $ is called a \textbf{semiring} if it satisfies the following properties:

 1. $ \emptyset \in \Sigma $,

 2. If $ A, B \in \Sigma$, then $A \cap B \in \Sigma$,

 3. For every $A, B \in \Sigma$, there exist $ C_{1}, ... , C_{n} \in \Sigma $ such that $A \setminus B =  \cup_{i=1}^{n} C_{i}$ and $ C_{i}\cap C_{j}= \emptyset$ if $ i \not=j $.

\textbf{Definition 2.2.} A nonempty collection $ \Sigma $ of subsets of $ X $ which is closed under finite intersections and complementation is called an \textbf{algebra} of sets. An algebra $\Sigma$ of subsets of $X$ is called a $\mathbf{\sigma}$\textbf{-algebra} if every union of a countable collection of elements of $\Sigma$ is again in $\Sigma$.

\textbf{Definition 2.3.}  Let $\Sigma$ be a semiring of subsets of  $X$. A set function \\ \ $ \mu : \Sigma  \rightarrow [0, \infty]$ is called a \textbf{measure} on $ \Sigma $ if it satisfies the following properties:

1. $\mu(\emptyset) = 0$ ,

2. If $ \lbrace A_{n} : n \in \nn \rbrace$ is a disjoint 
sequence in $ \Sigma  $ satisfying $ \cup_{n=1}^{\infty} 
A_{n} \in \Sigma$, then
\[ \mu ( \cup_{n=1}^{\infty} A_{n}) = \Sigma_{n=1}
^{\infty} \mu(A_{n})\] 
holds.

\textbf{Definition 2.4.} A triplet $ (X, \Sigma, \mu) $, where $ \Sigma $ is a semiring of subsets of $ X $ and $ \mu $ is a measure on $ \Sigma $ is called a \textbf{measure space} .

\textbf{Definition 2.5.} Let $ ( X, \Sigma, \mu) $ be a measure space and $ A \subseteq X $. If $ \mu(A) < \infty $, then $ A $ is said to be of \textbf{finite measure}. If $ A = \bigcup_{n \in \nn} A_{n} $, where each $ A_{n}$ is of finite measure, then $ A $ is said to be of $ \sigma $\textbf{-finite measure}.

\textbf{Definition 2.6.} A set function $\mu: P(X)\rightarrow [0, \infty]$ defined on the power set $\mathcal{P}(X)$ of $ X$ is called an \textbf{outer measure} if it satisfies these properties:

l. $\mu(\emptyset) = 0$ ,

2. If $A \subseteq B$, then $\mu(A) \leq \mu(B)$,

3. $ \mu( \cup_{n=1}^{\infty}A_{n}) \leq \Sigma_{n=1}^{\infty}\mu(A_{n})$ holds for every sequence $ (A_{n})_{n \in \nn}$ of subsets of X.

 Let $ \mathcal{F} $ be a collection of subsets of $ X $  containing the empty set. Also, let $ \mu: \mathcal{F} \rightarrow [0, \infty]$ be a set function such that  $\mu(\emptyset) = 0$. For every subset $ A $ of $ X $ we define $$\mu^{*}(A) = \inf \lbrace \Sigma_{n=1}^{\infty}\mu(A_{n}): (A_{n}) \text{ \rm is a sequence of } \mathcal{F} \text{ \rm with} A \subseteq \cup_{n=1}^{\infty}A_{n} \rbrace .$$ If there is no sequence $ (A_{n})_{n \in \nn} $ of $\mathcal{F}$  such that $A \subseteq \cup_{n=1}^{\infty}A_{n}$ , then we let $ \mu^{*}(A) = \infty$. 
 
The set function $\mu^{*}: P(X)\rightarrow [0, \infty]$ is an outer measure (called the \textbf{outer measure generated} by the set function $\mu : \mathcal{F}\rightarrow [0, \infty])$ satisfying 
$ \mu^{*}(A) \leq \mu(A) $ 
for each $A \in \mathcal{F}$.

\textbf{Definition 2.7.}  Let $ ( X, \Sigma_{X}, \mu_{X}) $ and
$ ( Y, \Sigma_{Y}, \mu_{Y} ) $ be measurable spaces. The subset $$ \Sigma_{X} \otimes \Sigma_{Y} := \lbrace A \times B : A \in \Sigma_{X}, B \in \Sigma_{Y} \rbrace$$
of $ X \times Y $ is called the \textbf{product semiring}. The set function 
$ \mu_{X}\times \mu_{Y} : \Sigma_{X} \otimes \Sigma_{Y}\rightarrow [0, \infty]$ defined by
$$\mu_{X}\times \mu_{Y}(A \times B) = \mu_{X}(A)\cdot\mu_{Y}(B)\ \ \text{\rm for all } A \times B \in \Sigma_{X} \otimes \Sigma_{Y}.$$
 is a measure on  $\Sigma_{X} \otimes \Sigma_{Y}$ and is called  the \textbf{product measure} of $ \mu_{X}$ and  $\mu_{Y} $ and is finitely additive.

\section{Product measures}

\begin{thm}

 Let $ ( X, \Sigma_{X}, \mu_{X}) $ and
$ ( Y, \Sigma_{Y}, \mu_{Y} ) $ be measurable spaces where $ \Sigma_{X}, \;  \Sigma_{Y}  $ are semi-rings ($ \sigma- $algebras) and $ \mu_{X} ,\; \mu_{Y} $ are $  \sigma- $ additive measures. Let $ D \subseteq X \times Y $ and suppose $ D \subseteq  \cup_{n \in \nn} B_{n} \times C_{n}$ where $ \lbrace B_{n} \times C_{n} : n \in \nn \rbrace$ are pairwise disjoint sequences with $  B_{n} \in \Sigma_{X} $ and $  C_{n} \in \Sigma_{Y} $. Define $ D^{>r}:= \lbrace x \in X : \mu_{Y}^{*}(D^{x}) > r \rbrace$ for $ r \in \rr^{+} $ and $ D^{x}:= \lbrace y : (x,y) \in D \rbrace $. Let $  \mu_{X}^{*},\; \mu_{Y}^{*}$ be outer measures generated by $\mu_{X}$ and $ \mu_{Y}$, respectively, and suppose $ \mu_{X}^{*}(D^{>r}) > s  $ for some $ s \in \rr^{+},$ Then, there exists a finite subset $ F \subseteq \nn $ such that $$rs < \underset{n \in F}\Sigma \mu_{X}(B_{n}) \mu_{Y}(C_{n}) $$.

\end{thm}

\begin{proof}

 Let $ x \in D^{>r} $ and define $ N_{x} := \lbrace n \in \nn : x \in B_{n} \rbrace$. Then, $ D^{x}\subseteq \bigcup_{n \in N_{x}} C_{n} $. Now
 $$ r < \mu_{Y}^{*}(D^{x}) \leq \Sigma_{n \in N_{x}}\mu_{Y}^{*}(C_{n}) \leq \Sigma_{n \in N_{x}}\mu_{Y}(C_{n}).$$

Let $ M_{x} $ be a finite subset of $ N_{x} \subseteq \nn $ such that $ r < \Sigma_{n \in M_{x}}\mu_{Y}(C_{n}) $. Note that $ x \in B_{n}$ for all $ n \in M_{x} $. Thus, $ D^{>r}\subseteq  \bigcup_{x \in  D^{>r} }(\underset{n \in M_{x}}\cap B_{n})   $ and the family  $\lbrace  \underset{n \in M_{x}}\cap B_{n} : x \in D^{>r}  \rbrace \subseteq \Sigma_{X}$ is countable since $ \lbrace M_{x} : M_{x}\subseteq \nn \rbrace \subseteq \mathcal{P}_{< \aleph_{o}}(\nn) .$

We have  $$ s <  \mu_{X}^{*}(D^{>r}) \leq \mu_{X}^{*}(  \bigcup_{x 
\in  D^{>r} }(\underset{n \in M_{x}}\cap B_{n})) . $$

 Thus, there exist a finite subset $ M $ of $ D^{>r}$ with
 $$ s < \mu_{X}^{*}( \underset{x \in M} \bigcup (\underset{n \in 
 M_{x}} \cap B_{n}) .$$ Set $ F := \bigcup_{x \in M}M_{x} $. Since $ 
 \underset{x \in M}\bigcup(\underset{n \in M_{x}}\cap B_{n} \times 
 \underset{n \in M_{x}}\cup C_{n}) \subseteq \underset{n \in F}
 \bigcup (B_{n} \times C_{n})$ we have \\
 $rs < \underset{n \in F}\Sigma \mu_{X}(B_{n}) \mu_{Y}(C_{n}) $ by the finite additivity of $ \mu_{X}\times \mu_{Y} $ on the semiring \\
  $\Sigma_{X} \otimes \Sigma_{Y}$.
 
\end{proof}

\coro Let $ ( X, \Sigma_{X}, \mu_{X}) $ and  $ ( Y, \Sigma_{Y}, \mu_{Y} ) $ be measurable spaces where $ \Sigma_{X}, \; \Sigma_{Y}  $ are semi-rings ($ \sigma- $algebras) and $ \mu_{X}$ and $ \mu_{Y} $ are $  \sigma- $ additive measures. Then  $\mu_{X} \times \mu_{Y}$ is countably additive on $ X \times Y $. 
\begin{proof}
 Let $( X, \Sigma_{X}, \mu_{X}) $ and $ ( Y, \Sigma_{Y}, \mu_{Y} ) $ be measurable spaces where $ \Sigma_{X} $  $ , \Sigma_{Y}  $ are semi-rings ($ \sigma- $algebras) and $ \mu_{X} ,\; \mu_{Y} $ are $  \sigma- $ additive measures. Let $ D = B \times C = \bigcup_{n \in \nn}(B_{n} \times C_{n})$ with  
$ B_{n}, B \in  \Sigma_{X} $ and $ C_{n}, C \in  \Sigma_{Y}$ for all $n \in \nn  $ and $ B_{n} \times C_{n} $ are pairwise disjoint for all $n \in \nn  $.

We may suppose, without loss of generality, $ \mu(B)>0$ and also,$ \mu(C)>0 .$

Suppose $ t< \mu_{X}(B) \cdot \mu_{Y}(C)$. We may find $ r,s \in 
\rr^{+} $ with $ 0 < r < \mu_{X}(B)$ and $ 0 < s < \mu_{Y}(C)$ such 
that $ t < r \cdot s $. Consider $ (B \times C)^{x} $ where $ x \in 
B. $ Clearly, \\ $ \mu^{*}((B \times C)^{x}) = \mu(C) $ and $ D^{>r}=B $. Thus, $ r < \mu(C) $ and $ s < \mu^{*}(D^{>r}) .$ By the theorem above, we may find a finite subset $ F $ of $ \nn $ such that $$ t < \underset{n \in F}\Sigma \mu(A_{n})\cdot \mu(B_{n}) .$$

\end{proof}

\begin{coro} Let $ ( X, \Sigma_{X}, \mu_{X}) $ and  $ ( Y, \Sigma_{Y}, \mu_{Y} ) $ be measurable spaces where $ \Sigma_{X}, \; \Sigma_{Y}  $ are semi-rings ($ \sigma- $algebras) and $ \mu_{X} ,\; \mu_{Y} $ are $  \sigma- $ additive measures. If $ D \subseteq X \times Y $
and $ (\mu_{X} \times \mu_{Y})^{*}(D) = 0 $ then $ \mu_{Y}^{*}(D^{x})= 0 $ for $ \mu_{X}$-almost all $ x \in X $. Conversely, if $ D $ is 
$ (\mu_{X} \times \mu_{Y})^{*} $ measurable and   $ \mu_{Y}^{*}(D^{x})= 0 $ for $ \mu_{X}$-almost all $ x$, then $ (\mu_{X} \times \mu_{Y})^{*}(D) = 0$, provided 
$ \mu_{X}, \; \mu_{Y} $ are $ \sigma$-finite measures. 
\end{coro}

\begin{proof}
 If $ D \subseteq X \times Y $
and $ (\mu_{X} \times \mu_{Y})^{*}(D) = 0 $ then $ \mu^{*}_{X} (D^{>r})= 0 $ for all $ r > 0$ by the theorem.  Hence $ \mu_{X}^{*}(\bigcup_{n \in \nn}D^{>\frac{1}{n}})= 0 $. Thus, $ \mu_{Y}^{*}(D^{x})= 0 $ for $ \mu_{X}$-almost all $ x \in X $.

 Conversely, let $ \mu_{X}, \; \mu_{Y} $ be $ \sigma$-finite measures. Suppose $ D $ is 
$ (\mu_{X} \times \mu_{Y})^{*} $ measurable and for $ \mu_{X}$-almost all $ x$,  $ \mu_{Y}^{*}(D^{x})= 0 $. Then $ (\mu_{X} \times \mu_{Y})^{*}(D) = 0 $ .

By the $ \sigma$-additivity of measures, we may assume $\mu_{X}(X), \; \mu_{Y}(Y) < \infty$. Suppose $ \mu_{Y}(D^{x}) = 0$ for $ \mu_{X}$-almost all $ x \in X $. Since $ D^{c}= X \times Y \setminus D $, we have $\mu_{Y}((D^{c})^{x})=  \mu_{Y}(Y)$ for $ \mu_{X}$-almost all $ x \in X $. Thus,
$$\mu_{X}^{*}(\lbrace x : \mu_{Y}((D^{c})^{x}) = \mu_{Y}(Y) \rbrace ) = \mu_{X}(X)
$$
Hence, by the theorem, $(\mu_{X} \times \mu_{Y})^{*}(D^{c})=  \mu_{X}(X)\cdot \mu_{Y}(Y)$. Therefore,\\ $ (\mu_{X} \times \mu_{Y})^{*}(D) = 0 $.

\vspace{5mm}

Note that both monotone and dominated convergence theorems themselves may be proved with this approach, which is, using measure theory instead of integration with a slightly different interpretation.

\end{proof}
\end{document}